\documentclass[letterpaper,12pt,normalheadings]{scrartcl}

\usepackage[american]{babel}
\usepackage[utf8]{inputenc}
\usepackage[T1]{fontenc}
\usepackage[shortcuts]{extdash}
\usepackage{lmodern}
\usepackage{geometry}
\usepackage{parskip}
\usepackage{caption}
\usepackage{subcaption}
\usepackage{authblk}
\usepackage{hyperref}
\hypersetup{colorlinks=true, linkcolor=chroma4blue, citecolor=chroma4blue, urlcolor=black, filecolor=black}
\usepackage{graphicx}
\usepackage{amsmath}
\usepackage{amssymb}
\usepackage{amsthm}
\usepackage{dsfont}
\usepackage{mathtools}
\usepackage{booktabs}
\usepackage{tabulary}
\usepackage{xspace}
\usepackage{tikz}
\usepackage{nicematrix}

\definecolor{chroma4blue}{RGB}{63,97,135}
\definecolor{chroma4green}{RGB}{65,143,126}
\definecolor{chroma4sand}{RGB}{171,166,125}
\definecolor{chroma4gray}{RGB}{199,199,199}

\NiceMatrixOptions{transparent, xdots/shorten=0.6em}

\mathcode`@="8000
{\catcode`\@=\active\gdef@{\mkern1mu}}

\DeclareMathOperator{\diag}{diag}

\newcommand{\Ep}{{\cal E}}

\newcommand{\R}{\mathbb{R}}

\newcommand{\Cm}{C}
\newcommand{\cm}{c}
\newcommand{\Pm}{P}
\newcommand{\GH}{Gomory-Hu triangle inequality\xspace}

\newcommand{\bo}{\mathds{1}}

\theoremstyle{definition}
\newtheorem{lemma}{Lemma}
\newtheorem*{example*}{Example}

\newtheorem{corollary}[lemma]{Corollary}
\newtheorem{theorem}[lemma]{Theorem}

\newtheorem{definition}[lemma]{Definition}

\numberwithin{lemma}{section}
\numberwithin{fact}{section}
\numberwithin{equation}{section}

\makeatletter
\newcommand{\pushright}[1]{\ifmeasuring@#1\else\omit\hfill$\displaystyle#1$\fi\ignorespaces}
\newcommand{\pushleft}[1]{\ifmeasuring@#1\else\omit$\displaystyle#1$\hfill\fi\ignorespaces}
\makeatother

\makeatletter
\newlength{\negph@wd}
\DeclareRobustCommand{\negphantom}[1]{%
  \ifmmode
    \mathpalette\negph@math{#1}%
  \else
    \negph@do{#1}%
  \fi
}
\newcommand{\negph@math}[2]{\negph@do{$\m@th#1#2$}}
\newcommand{\negph@do}[1]{%
  \settowidth{\negph@wd}{#1}%
  \hspace*{-\negph@wd}%
}
\makeatother

\addtokomafont{sectioning}{\rmfamily}

\title{\Large \textbf{Edge-connectivity matrices and their spectra}}
\author{\normalsize Tobias Hofmann, Uwe Schwerdtfeger}
\date{} 
\affil{\normalsize
\texttt{tobias.hofmann@math.tu-chemnitz.de}\\\normalsize \texttt{regeftdrewhcs.ewu@gmail.com}\\
}

\begin{document}
\maketitle

\begin{abstract}
\noindent
\textbf{Abstract.} The edge-connectivity matrix of a weighted graph is the matrix whose off-diagonal $v$-$w$ entry is the weight of a minimum edge cut separating vertices $v$ and $w$. Its computation is a classical topic of combinatorial optimization since at least the seminal work of \mbox{Gomory} and Hu. In this article, we investigate spectral properties of these matrices. In particular, we provide tight bounds on the smallest eigenvalue and the energy. Moreover, we study the eigenvector structure and show in which cases eigenvectors can be easily obtained from matrix entries. These results in turn rely on a new characterization of those nonnegative matrices that can actually occur as edge-connectivity matrices.\\



\end{abstract}

\section{Introduction}

For an undirected graph~$G=(V,E)$ with nonnegative edge weights its \emph{edge-connectivity matrix} is the $V\times V$ matrix $\Cm(G)$ whose off-diagonal $v$-$w$ entry denotes the minimum weight of an edge set whose removal disconnects the vertices $v$~and~$w$. The diagonal entries are defined as zero. In the seminal article~\cite{gomoryhu1961multi}, Gomory and Hu show that there exists a weighted tree~$T=(V,F)$ on the same vertex set as~$G$, but not necessarily with $F\subseteq E$, such that $\Cm(T)=\Cm(G)$. Even stronger, for each pair of vertices $v$~and~$w$ the two sides of a minimum cut that separates $v$~and~$w$ in~$T$ also induce a minimum cut that separates $v$~and~$w$ in~$G$. An auxiliary result of Gomory and Hu, which is particularly important for our investigation, is a characterization of those matrices that can occur as edge-connectivity matrices of weighted graphs. This description is in terms of a special triangle inequality, which is stated in Theorem~\ref{thm:triangle}. In this article, we provide another characterization and demonstrate how it can be utilized to gain further insights about the spectrum.
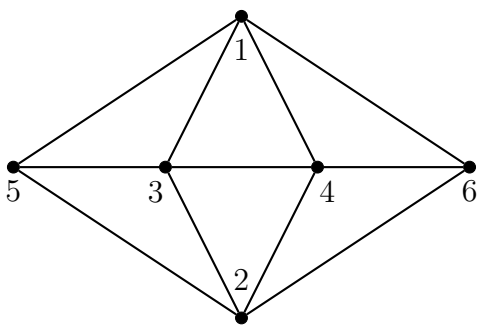
\begin{figure}
\centering
\begin{tikzpicture}[scale=2]
    \node[draw=black, fill=black, inner sep=0.55mm, circle, label={[label distance=-0.25ex]270:$5$}] (5) at (0,0) { };
    \node[draw=black, fill=black, inner sep=0.55mm, circle] (3) at (1,0) { };
    \node[label={270:$3$}] (3') at (0.935,0.25ex) { };
    \node[draw=black, fill=black, inner sep=0.55mm, circle] (4) at (2,0) { };
    \node[label={270:$4$}] (4') at (2.065,0.25ex) { };
    \node[draw=black, fill=black, inner sep=0.55mm, circle, label={[label distance=-0.25ex]270:$6$}] (6) at (3,0) { };
    \node[draw=black, fill=black, inner sep=0.55mm, circle, label={[label distance=0.5ex]270:$1$}] (1) at (1.5,1) { };
    \node[draw=black, fill=black, inner sep=0.55mm, circle, label={[label distance=0.75ex]90:$2$}] (2) at (1.5,-1) { };
    \draw[black, thick] (1) to (3);
    \draw[black, thick] (1) to (4);
    \draw[black, thick] (1) to (5);
    \draw[black, thick] (1) to (6);
    \draw[black, thick] (2) to (3);
    \draw[black, thick] (2) to (4);
    \draw[black, thick] (2) to (5);
    \draw[black, thick] (2) to (6);
    \draw[black, thick] (5) to (3);
    \draw[black, thick] (3) to (4);
    \draw[black, thick] (4) to (6);
\end{tikzpicture}
\caption{A graph $G$ for which the matrix $P(G)+D(G)$ is not positive semidefinite}
\label{fig:CounterExample}
\end{figure}%

Our considerations about the spectrum of the edge-connectivity matrix are inspired by recent articles that investigate the spectrum of the \emph{vertex-connectivity matrix} or \emph{path matrix}~$\Pm(G)$ of an unweighted graph $G=(V,E)$. This is the $V\times V$~matrix whose off-diagonal \mbox{$v$-$w$}~entry is the maximum number of independent $v$-$w$~paths with only zeros on the diagonal. Shikare and coauthors~\cite{shikare2018path} raised the conjecture that the \mbox{\emph{energy}} of~$\Pm(G)$, that is the sum of the absolute values of the eigenvalues, is at most~$2(n-1)^2$. Ilić and Bašić~\cite{ilic2019path} claim to prove this bound by employing the approach of Koolen and Moulton~\cite{koolen2001maximal}. However, following their arguments carefully, they only obtain an upper bound of $(n-1)^{5/2}+(n-1)^{3/2}$, which is strictly larger than~$(n-1)^2$ for all~$n\in\mathbb{N}$. Nevertheless, our numerical investigations have not yet revealed any counterexamples to the stated bound. An analogous result for the edge-connectivity matrix is presented in Item~\ref{thm:VDPassertions:1} of Theorem~\ref{thm:VDPassertions}. In another article~\cite{patekar2020signless}, Patekar and Shikare claim that $\Pm(G)+D(G)$ is positive semidefinite, where $D(G)$ denotes the diagonal matrix of vertex degrees. This would immediately imply the conjecture about the energy. However, the claim that $\Pm(G)+D(G)$ is positive semidefinite is indeed false. A counter example is given by the graph $G$ in Figure~\ref{fig:CounterExample}
for which
\begin{equation*}
    \Pm(G)+D(G) = \begin{bmatrix}[first-row, first-col]
          & \textcolor{chroma4gray!80!black}{1} & \textcolor{chroma4gray!80!black}{2} & \textcolor{chroma4gray!80!black}{3} & \textcolor{chroma4gray!80!black}{4} & \textcolor{chroma4gray!80!black}{5} & \textcolor{chroma4gray!80!black}{6} \\
        \textcolor{chroma4gray!80!black}{1} & 4 & 4 & 4 & 4 & 3 & 3 \\
        \textcolor{chroma4gray!80!black}{2} & 4 & 4 & 4 & 4 & 3 & 3 \\
        \textcolor{chroma4gray!80!black}{3} & 4 & 4 & 4 & 3 & 3 & 3 \\
        \textcolor{chroma4gray!80!black}{4} & 4 & 4 & 3 & 4 & 3 & 3 \\
        \textcolor{chroma4gray!80!black}{5} & 3 & 3 & 3 & 3 & 3 & 3 \\
        \textcolor{chroma4gray!80!black}{6} & 3 & 3 & 3 & 3 & 3 & 3 
    \end{bmatrix}.
\end{equation*}
This matrix is not positive semidefinite as it has a negative principal minor 
\begin{equation*}
    \det
    \begin{bmatrix}[r]
        4 & \hphantom{\!\!-}4 & \hphantom{\!\!-}4 \\
        4 & 4 & 3 \\
        4 & 3 & 4
    \end{bmatrix} = \det
    \begin{bmatrix}[r]
        4 & 0 & 0 \\
        4 & 0 & \!\!-1 \\
        4 & \!\!-1 & 0
    \end{bmatrix} = 4 \det
    \begin{bmatrix}[r]
        0 & \!\!-1 \\
        -1 & 0
    \end{bmatrix} = -4.
\end{equation*}
Another attempt to prove the energy bound was to show the weaker assumption that the smallest eigenvalue of $\Pm(G)$ is at least $-(n-1)$, but computer search revealed counterexamples, even for a lower bound $-(n+1)$.

It appears rather natural to consider also the related problem where the number of independent paths is replaced by the number of edge-disjoint paths or, what is the same, edge cuts. The preceding equivalence is the edge version of Menger's theorem~\cite{menger1927allgemeinen}, which, for instance, is presented by Diestel~\cite[Section 3]{diestel2017graph}. It turns out that the aforementioned assertions are both true for the edge-connectivity matrix and even stronger statements hold.
\begin{theorem}
For a weighted graph $G$ on vertex set $V$ let $\Cm=(c_{vw})$ be its edge-connectivity matrix. Denote for $v\in V$ by $m(v)\coloneqq\max\{\cm_{vw} : w\in V\setminus\{v\} \}$ the maximum off-diagonal entry in row $v$ and by $M\coloneqq\max\{\cm_{vw} : v,w\in V\}$ the maximum entry of~$\Cm$. Then the following statements hold.
\begin{enumerate}\label{thm:VDPassertions}
    \item The matrix $\Cm+\diag(m(v) : v\in V)$ is positive semidefinite.\label{thm:VDPassertions:1}
    \item The smallest eigenvalue of $\Cm$ is $-M$.\label{thm:VDPassertions:2}
    \item The energy of $\Cm$ is at most $2(n-1)M$ with equality if and only if $G$ is uniformly $M$-edge-connected.\label{thm:VDPassertions:3}
    \end{enumerate}
\end{theorem}
Note that Item~\ref{thm:VDPassertions:1} comprises the result that $\Cm(G)+D(G)$ is positive semidefinite for any graph $G$ because all $v,w\in V(G)$ satisfy $\cm_{vw}\le\min\{\deg(v),\deg(w)\}$ and thus also $m(v)\leq\deg(v)$ is true for all $v\in V(G)$. Proofs for the items of Theorem~\ref{thm:VDPassertions} are given in the subsequent sections in a more general setting.

Basic properties of the edge-connectivity matrix are also investigated by Akbari and coauthors~\cite{akbari2020edge}. They determine the spectrum of bicyclic graphs and consider the structure of \emph{uniformly $k$-edge-connected} graphs. These are graphs in which~\mbox{$k\in \mathbb{N}$} is the maximum number of edge-disjoint paths between any two vertices. Analogously, a graph is called \emph{uniformly $k$-connected} if \mbox{$k\in\mathbb{N}$} is the maximum number of independent paths between any two vertices. More about the structure of these graph classes can be found in Göring, Hofmann, and Streicher~\cite{uniformly2021}. Also note that there are results on the largest eigenvalue $\varrho$ of vertex-connectivity matrices. Shikare and coauthors~\cite{shikare2018path} show that $(n-1)\leq\varrho\leq(n-1)^2$ for a graph on $n$ vertices by using Perron-Frobenius arguments, which are presented comprehensively by Horn and Johnson~\cite[Chapter 8]{horn1990matrix}. Both bounds are tight. They are attained for trees or complete graphs, respectively. Moreover, the same arguments and obtained bounds hold analogously for edge-connectivity matrices.

{\bf Outline.} In Section \ref{sec:characterization}, we present an alternative characterization of edge-connec\-tivity matrices. We use these results in Section~\ref{sec:spectral} to prove spectral bounds and eigenvector properties. Finally, Section~\ref{sec:ultrametric} provides an alternative interpretation of our results for distance matrices whose entries satisfy an ultrametric.

We conclude this introduction with further notation that is particularly important for our investigation. All vectors and matrices in this article are indexed by some finite set~$V$ or~$V\times V$, respectively, and we write $n=|V|$ for short. For a subset of $X\subseteq V$ we write $\bo_X$ for the vector $x$ with~$x_v = 1$ if $v\in X$ and $x_v=0$ otherwise. By $J_X$ we denote the $V\times V$ matrix for which $(J_X)_{vw}=1$ if $(v,w)\in X\times X$ and $0$ otherwise. In other words, $J_X=\bo_{X}\bo_{X}^\top$. By $\bo$ we mean the all ones vector, for $J_V$ we occasionally write $J$ for short and~$I$ is the identity matrix. Furthermore, we denote $V\times V$ diagonal matrices whose diagonal entries are given by a sequence $(x_v)_{v\in V}$ by $\diag(x_v : v\in V)$. The vectors of the standard basis of $\R^V$ are denoted by $e_v$ for $v\in V$. For a symmetric matrix~$\Cm$ with eigenvalues $\lambda_1,\ldots,\lambda_n$ its \emph{energy}~$\Ep(\Cm)$ is defined as $\Ep(\Cm)=\sum_{i=1}^n |\lambda_i|$. An overview about methods and different variants of the energy concept is given by Li, Shi, and Gutman in the monograph~\cite{gutman2012energy}. By an \emph{equitable partition} of a matrix~$\Cm$ we mean a partition of~$\Cm$ into submatrices
\begin{equation*}
    \Cm= \begin{bmatrix}
        \Cm^{11} & \cdots & \Cm^{1k} \\
        \vdots   & \ddots & \vdots   \\
        \Cm^{k1} & \cdots & \Cm^{kk}
\end{bmatrix},
\end{equation*}
where each block $\Cm^{ij}$ has constant row sums equal to $q_{ij}$. The matrix $Q=(q_{ij})$ is called an \emph{equitable quotient matrix} of $\Cm$ and all eigenvalues of $Q$ are also eigenvalues of $\Cm$. You and coauthors~\cite{you2019spectrum} provide details about the structure of such matrices. Lastly, for graph theoretical terminology we refer to Diestel~\cite{diestel2017graph}. 

\section{Characterizing edge-connectivity matrices}\label{sec:characterization}
To begin, we recall the following characterization that is due to Gomory and Hu~\cite{gomoryhu1961multi}.
\begin{theorem}\label{thm:triangle}
A nonnegative real symmetric $V\times V$ matrix $\Cm$ with only zeros on its diagonal occurs as an edge-connectivity matrix of a weighted graph if and only if it satisfies the \emph{\GH}
\begin{equation*}\label{eq:property1}
    \cm_{xz} \ge \min\{\cm_{xy}, \cm_{yz}\}
\end{equation*}
for all $x,y,z \in V$ with $x\neq y\neq z\neq x$.
\end{theorem}
Our characterization is in terms of the following property.
\begin{definition}\label{def:terraced}
For a real symmetric $V\times V$ matrix~$\Cm=(\cm_{vw})$ and a real number~$\ell$ we denote by
\begin{equation*}\label{eq:def:superlevelsets}
    S_\ell(\Cm)=\{(v,w)\in V\times V:\cm_{vw}\ge \ell \}
\end{equation*}
the \emph{superlevel set} of $\Cm$. We call $\Cm$ \emph{terraced} if for each $\ell\in\R$ there is a set $\mathcal{T}(\ell)$ of pairwise disjoint subsets of $V$ such that
\begin{equation*}
    S_\ell(\Cm) = \,\bigcup_{\mathclap{X \in \mathcal{T}(\ell)}}\, X\times X 
\end{equation*}
\end{definition}
Both the \GH as well as the terraced structure make perfect sense without any assumptions on nonnegativity. So we formulate the next theorem for arbitrary symmetric matrices. 
\begin{theorem}\label{thm:characteriztion}
For a real symmetric $V\times V$ matrix $\Cm$ the following statements are equivalent.
\begin{enumerate}
    \item The matrix $\Cm$ is terraced.\label{thm:characteriztion:1}
    \item The matrix $\Cm$ satisfies the Gomory-Hu triangle inequality for all $x,y,z \in V$. In particular, $\cm_{vv}\ge \cm_{vw}$ for all $v, w\in V$.\label{thm:characteriztion:2}
    \end{enumerate}
\end{theorem}
\begin{proof}
We show first that \ref{thm:characteriztion:1} implies \ref{thm:characteriztion:2}. So let $\Cm$ be terraced and define $\mathcal{T}(\ell)$ as in Definition~\ref{def:terraced}. We choose $x,y,z\in V$ without loss of generality such that $\min\{\cm_{xy}, \cm_{yz}\}=\cm_{yz}\eqqcolon \ell$, in other words $\cm_{xy}\ge \cm_{yz}\ge \ell$. Then there is a set $X$ in~$\mathcal{T}(\ell)$ with $(x,y)\in X\times X$, which means that $x,y\in X$. Likewise, there is a set $X^\prime$ with $y,z\in X^\prime$. Because $z\in X\cap X^\prime$, we obtain that $X=X^\prime$, as the sets in $\mathcal{T}(\ell)$ are defined to be pairwise disjoint. We observe that $x,y,z\in X$ and in particular $(x,z)\in X\times X$ which implies $\cm_{xz}\ge \ell$, as desired.

We show that \ref{thm:characteriztion:2} implies \ref{thm:characteriztion:1} by induction on $n\coloneqq|V|$. For $n=1$ there is nothing to show. So let $n>1$ and define $L\coloneqq\min\{\cm_{vw} : v,w \in V\}$. Suppose that the column with index $z$ contains an entry equal to $L$ and define two sets $X$ and $Y$ by
\begin{equation*}
    X=\{x\in V: x\neq z \text{ and }\cm_{xz} = L\} \text{ and } Y=V\setminus X.
\end{equation*}
Observe that $X\neq\emptyset$ because $\cm_{zz}\geq\cm_{vz}$ for $v\in V$. Furthermore, we have $Y\neq\emptyset$ because $z\in Y$. For $x\in X$ and $y\in Y$ it follows that $\cm_{xz}\leq \cm_{yz}=\cm_{zy}$ and equality can only hold if $Y=\{z\}$ and $\cm_{zz}=L$. Applying the \GH twice, we obtain that
\[
\cm_{xy} \geq \min\{\cm_{xz}, \cm_{zy}\} = \cm_{xz}\geq \min\{\cm_{xy}, \cm_{yz}\} = \cm_{xy}.
\]
The last equality deserves a word of explanation. If $y\neq z$, then $\cm_{xz}<\cm_{yz}$, as otherwise $x\notin X$. Thus for the preceding inequality to hold, the last minimum must be equal to $\cm_{xy}$. If $y=z$, the last equality follows from $\cm_{zz}\ge \cm_{xz}$, which is true by the assumption in~\ref{thm:characteriztion:2}. In both cases, we obtain that $\cm_{xy}=\cm_{xz}=L$ for an arbitrary pair $(x,y)\in X\times Y$.

So we observe that all entries of $\Cm$ restricted to $X\times Y$ and $Y\times X$ are equal to $L$. Consequently, for $\ell>L$ the superlevel set $S_\ell(\Cm)$ is contained in $(X\times X) \cup (Y\times Y)$. The restrictions $\Cm\rvert_{X\times X}$ and $\Cm\rvert_{Y\times Y}$ of $\Cm$ to $X\times X$ and $Y\times Y$, respectively, satisfy the triangle inequality and, by induction, are terraced. Because $S_L(\Cm) = V$ and
\begin{equation*}
    S_\ell(\Cm) = S_\ell\big(\Cm\rvert_{X\times X}\big)\cup S_\ell\big(\Cm\rvert_{Y\times Y}\big)\quad\text{for }\ell>L,
\end{equation*}
we conclude that $\Cm$ is of the desired form.
\end{proof}
In the specific case of an edge-connectivity matrix $\Cm$ of a graph~$G$ on $n$ vertices, we know by the results of Gomory and Hu~\cite{gomoryhu1961multi} that such a matrix can have at most~$n-1$ different values for its off-diagonal entries. Thus a terraced matrix can have at most~$2n-1$ distinct values.

\section{Spectral properties of terraced matrices}\label{sec:spectral}
With the structural results of the previous section at hand, we proceed with spectral investigations, which enable us to resolve the claims from Theorem~\ref{thm:VDPassertions}.
\begin{theorem}\label{thm:psd}
A real nonnegative terraced matrix is positive semidefinite.
\end{theorem}
\begin{proof}
Let $0 \le \ell_0 < \dots < \ell_k$ be the distinct values of entries of a real nonnegative terraced matrix $\Cm$. With the notation of Definition~\ref{def:terraced}, we can write $\Cm$ as
\begin{equation*}
    \Cm= \ell_0@J_V + \sum_{\mathclap{i=1}}^k~\sum_{\mathclap{~~~~X\in \mathcal{T}(\ell_i)}}\,(\ell_i-\ell_{i-1})@J_X.
\end{equation*}
This is a nonnegative linear combination of positive semidefinite matrices which is why $\Cm$ is positive semidefinite as well. 
\end{proof}
\begin{proof}[Proof of Item~\ref{thm:VDPassertions:1} of Theorem~\ref{thm:VDPassertions}]
Let $\Cm=(\cm_{vw})$ be the edge-connectivity matrix of a weighted graph $G$. Then $C$ satisfies the \GH by Theorem~\ref{thm:triangle}. Furthermore, the matrix $\Cm+\diag(m(v) : v\in V)$ by definition satisfies $\cm_{vv}\ge \cm_{vw}$ for all $v, w\in V$, because $m(v)$ denotes the maximum entry in \mbox{row~$v\in V$}. So Theorem~\ref{thm:characteriztion} tells us that $\Cm+\diag(m(v) : v\in V)$ is terraced. Moreover, it is clearly real and nonnegative and thus positive semidefinite, by Theorem~\ref{thm:psd}.
\end{proof}
Some eigenvalues and eigenvectors of terraced matrices can be read off from row maxima directly. In particular, we easily obtain the smallest eigenvalue in case of a nonnegative terraced matrix.
\begin{theorem}\label{thm:mineig}
Let $\Cm=(c_{vw})$ be a real symmetric matrix whose off-diagonal entries satisfy the \GH and let $x,y\in V$ with $x\neq y$. Then the value $\cm_{xy}$ is maximal among the off-diagonal elements in its row and column and $\cm_{xx}=\cm_{yy}$ if and only if $e_x-e_y$ is an eigenvector of $\Cm$ with corresponding eigenvalue $\cm_{xx}-\cm_{xy}$.
\end{theorem}
\begin{proof}
First, let $x,y\in V$ with $x\neq y$ satisfy $\cm_{xx}=\cm_{yy}$ and let $\cm_{xy}=\cm_{yx}$ be a maximum off-diagonal entry in its row and column. As $\Cm$ satisfies the \GH, we obtain for each $z\in V$ with $z\neq x$ and $z\neq y$ that
\[
\cm_{xz}\ge \min\{\cm_{xy},\cm_{yz}\}=\cm_{yz}\ge \min\{\cm_{yx},\cm_{xz}\} = \cm_{xz}.
\]
This implies that $\cm_{xz}=\cm_{yz}$. As a consequence,
\[
\big( \Cm@(e_x-e_y) \big)_z=\begin{cases}
    0                                       &\text{if }z\neq x \text{ and }z\neq y,\\
    \cm_{xx}-\cm_{xy}                       &\text{if }z=x,\\
    \cm_{yx}-\cm_{yy} = \cm_{xy}-\cm_{xx}   &\text{if }z=y
\end{cases}
\]
and thus $e_x-e_y$ is an eigenvector for $\cm_{xx}-\cm_{xy}$.

Conversely, let $e_x - e_y$ be an eigenvector of $\Cm$ with corresponding eigenvalue $c_{xx}-c_{xy}$ for some~$x,y\in V$. Then the eigenequation $\left(\Cm(e_x-e_y) \right)_z=\cm_{zx}-\cm_{zy}=0$ for $z\in V$ with $z\neq x$ and $z\neq y$ implies that $\cm_{zx}=\cm_{zy}$. Because $\Cm$ is symmetric it follows that $\cm_{xz}=\cm_{zy}=\min\{\cm_{xz},\cm_{zy}\}\leq \cm_{xy}$ by the \GH. Hence, $\cm_{xy}$ is the maximum off-diagonal entry in its row and column. Furthermore, as $\cm_{xx}-\cm_{xy}$ is an eigenvalue of $\Cm$ with eigenvector $e_x - e_y$, the eigenequation $\left(\Cm(e_x-e_y) \right)_y=\cm_{yx}-\cm_{yy} = \cm_{xy}-\cm_{xx}$ implies that $\cm_{xx}=\cm_{yy}$, as desired.
\end{proof}
\begin{proof}[Proof of Item~\ref{thm:VDPassertions:2} of Theorem~\ref{thm:VDPassertions}]
Let $C=(\cm_{vw})$ be the edge-connectivity matrix of a weighted graph. Then $\Cm$ is nonnegative with only zeros on its diagonal and its off-diagonal entries satisfy the \GH by Theorem~\ref{thm:triangle}. Elements $x,y\in V(G)$ with $\cm_{xy} = M\coloneqq \max\{\cm_{vw} : v,w\in V\}$ satisfy the conditions of Theorem~\ref{thm:mineig} and so $\cm_{xx}-\cm_{xy}=-M$ is an eigenvalue of $\Cm$. Moreover, by Theorem \ref{thm:characteriztion}, $MI+\Cm$ is a nonnegative terraced matrix all of whose eigenvalues are nonnegative according to Theorem~\ref{thm:psd}. So $-M$ is indeed the smallest eigenvalue of $\Cm$.
\end{proof}
The energy bound~\ref{thm:VDPassertions:3} of Theorem \ref{thm:VDPassertions} follows directly from the following more general statement.
\begin{theorem}\label{thm:energy}
Let $\Cm$ be a real nonnegative symmetric matrix whose off-diagonal entries satisfy the \GH with $n$ rows and columns and only zeros on its diagonal. Then the energy of $C$ is at most $2(n-1)M$, where $M\coloneqq\max\{\cm_{vw} : v, w\in V\}$. Furthermore, $M(J-I)$ is the only such matrix that attains this bound.
\end{theorem}
\begin{proof}
Because the trace of $C$ is zero, the energy of $C$ is twice the sum of the absolute values of the negative eigenvalues. This sum is at most $(n-1)$ times the absolute value of the smallest eigenvalue which by Theorem \ref{thm:mineig} is $-M$.

It is an easy computation that the matrix $M(J-I)$ attains the bound. Furthermore, it is a consequence of the Perron-Frobenius theory that any other matrix whose entries are bounded from above by $M$ has a largest eigenvalue strictly less than $(n-1)M$. So $M(J-I)$ is indeed the only matrix which attains the bound.
\end{proof}
In terms of graph theory, the preceding energy bound is attained if and only if we are given a uniformly $M$-edge-connected graph. In the remainder of this section, we refine our results on the eigenvector structure of edge-connectivity matrices and achieve a lower bound for the energy.
\begin{theorem}\label{thm:energyRefined}
Let $\Cm=(\cm_{xy})$ be a symmetric $V\times V$ matrix whose off-diagonal entries satisfy the \GH and denote the maximum off-diagonal entry in row $v\in V$ by $m(v)\coloneqq\max\{\cm_{vw} : w\in V\setminus \{v\} \}$. Then the following statements hold.
\begin{enumerate}
    \item The matrix $\Cm$ induces an equivalence relation on $V$ by\label{thm:energyRefined:1}
    \[
    x \sim y :\Leftrightarrow x = y\text{ or }m(x) = m(y) = \cm_{xy}.
    \]
    \item If $x_1\sim x_2$ and $y_1\sim y_2$, then $\cm_{x_1y_1}=\cm_{x_2y_2}$ or, equivalently, an entry $\cm_{xy}$ depends only on the equivalence classes of $x$ and $y$.\label{thm:energyRefined:2}
    \item Let $X_1,\ldots, X_k$ be the equivalence classes with respect to the relation~$\sim$ and assume further that the diagonal entries satisfy $\cm_{xx}=\cm_{yy}$ whenever $x\sim y$. Then the restrictions $\Cm\rvert_{X_i\times X_j}$ of $\Cm$ to $X_i\times X_j$ induce an equitable partition of $\Cm$ with equitable quotient matrix $Q=(q_{ij})_{i,j=1,\ldots,k},$ where
    \[
    q_{ij} =\begin{cases}
        \cm_{xy}@|X_j|          & \text{if }i\neq j,\text{ where }x\in X_i \text{ and } y\in X_j, \\
        \cm_{xx}+m(x)(|X_j|-1)  & \text{if }i=j,\text{ where }x\in X_j.
    \end{cases}
    \]\label{thm:energyRefined:3}
    \item Assume in addition to~\ref{thm:energyRefined:3} that $\Cm$ is nonnegative and has only zeros on its diagonal. Denote further by $m(X_i)$ the common value $m(x)$ for $x\in X_i$. Then
    \begin{equation*}
    \Ep(\Cm) = \Ep(Q) + \textnormal{trace}(Q)  \ge 2\sum_{\mathclap{i=1}}^k (|X_{i}|-1)@m(X_i).
\end{equation*}
This inequality is an equality if and only if the equitable quotient matrix $Q$ has no negative eigenvalues. \label{thm:energyRefined:4}
    \end{enumerate}
\end{theorem}
\begin{proof}
We begin with Item~\ref{thm:energyRefined:1}. The relation $\sim$ is reflexive by definition. It is symmetric because $\Cm$ is. As for transitivity assume that $x\sim y$ and $y\sim z$. Hence, $m(x)=\cm_{xy}=m(y)=\cm_{yz}=m(z)$. Thus we obtain by applying the \GH that $\cm_{xy}=\min\{\cm_{xy},\cm_{yz}\}\le \cm_{xz}\le m(x)=\cm_{xy}$ and hence $\cm_{xz}=\cm_{xy}=m(x)=m(z)$, as desired.

For Item~\ref{thm:energyRefined:2}, let $x_1\sim x_2$ and $y_1\sim y_2$. This means that $\cm_{x_1x_2}$ is the maximum entry in rows $x_1$ and $x_2$. Likewise, $\cm_{y_1y_2}$ is the maximum entry in rows $y_1$ and $y_2$. By applying the \GH and the symmetry of $C$ repeatedly, we obtain that
\begin{align*}
    \cm_{x_1y_1}&\geq\min\{\cm_{x_1y_2},\cm_{y_2y_1}\}=\cm_{x_1y_2}\geq\min\{\cm_{x_1x_2},\cm_{x_2y_2}\}=\cm_{x_2y_2} \\
    &\geq\min\{\cm_{x_2y_1},\cm_{y_1y_2}\}=\cm_{x_2y_1}\geq\min\{\cm_{x_2x_1},\cm_{x_1y_1}\}=\cm_{x_1y_1}.
\end{align*}
We observe that all inequalities in this chain are indeed equalities and consequently $\cm_{x_1y_1}=\cm_{x_2y_2}$, as desired.

For Item~\ref{thm:energyRefined:3}, note that Item~\ref{thm:energyRefined:2} provides us with the fact that the submatrices $\Cm\rvert_{X_i\times X_j}$ are constant for $i,j\in V$ with $i\neq j$. This implies in particular that their row sums are constant. Similarly, for $i\in V$ all off-diagonal entries of $\Cm\rvert_{X_i\times X_i}$ have the same value by the definition of the relation $\sim$ and thus all diagonal entries have the same value by assumption. The stated formula for the entries $q_{ij}$ is a direct consequence.

We finally turn to Item~\ref{thm:energyRefined:4}. Let $X=\{x_1,\ldots,x_s\}$ be an equivalence class with respect to the relation $\sim$. Then we obtain $s-1$ linearly independent eigenvectors $e_{x_1}-e_{x_i}$ for $i=2,\ldots, s$ with a corresponding eigenvalue $-m(X)\le 0$. This contributes $-(s-1)m(X)$ to the sum of the negative eigenvalues. Therefore, summation over all classes, contributes
\begin{equation*}
    \sum_{i=1}^k \big((|X_{i}|-1)@m(X_i)\big)=\textnormal{trace}(Q)
\end{equation*}
to $\Ep(\Cm)$. From the partition of $V$ into equivalence classes $X_1,\ldots, X_k$, we obtain in total $|V|-k$ nonpositive eigenvalues, when counting multiplicities. The corresponding linearly independent eigenvectors are of the form $e_x-e_y$ where $x\sim y$. The remaining $k$ eigenvectors in an eigenbasis can be chosen orthogonal on the aforementioned vectors. This means, they can be chosen constant on the classes $X_i$ or, equivalently, of the form
\[
Z=\sum_{i=1}^k z_i\bo_{X_i}\quad\text{for appropriate }z_i\in \R.
\]
The corresponding eigenequation $\Cm Z=\lambda Z$ is equivalent to $Qz=\lambda z$, where $z=(z_1,\ldots,z_k)^\top$. This shows that the remaining $k$ eigenvalues of $\Cm$, in particular the positive ones, are among the eigenvalues of $Q$. This provides us with the relation~$\Ep(\Cm) = \Ep(Q) + \textnormal{trace}(Q)$ which in turn implies the bound to be shown, as clearly $\Ep(Q) \ge \textnormal{trace}(Q)$ with equality if and only if all eigenvalues of $Q$ are nonnegative.
\end{proof}

\begin{corollary}
The equitable quotient matrix $Q$ from the previous theorem is similar to the symmetric matrix $Q^\prime=(q^\prime_{ij})$ with
\[
    q^\prime_{ij} =\begin{cases}
        \cm_{xy}@(|X_j||X_i|)^{1/2}          & \text{if }i\neq j,\text{ where }x\in X_i \text{ and } y\in X_j, \\
        \cm_{xx}+m(x)(|X_j|-1)  & \text{if }i=j,\text{ where }x\in X_j.
    \end{cases}
    \]
If $Q^\prime$ is positive semidefinite, the inequality in the previous theorem is an equality.
\end{corollary}
\begin{proof}
We have $Q^\prime = WQW^{-1}$ with $W=\diag(|X_i|^{1/2}:i=1,\ldots,k).$
\end{proof}

\begin{example*}
The edge-connectivity matrix of the graph $G$ in Figure~\ref{fig:CounterExample} is
\begin{equation*}
    \Cm(G) = \begin{bmatrix}[first-row, first-col]
          & \textcolor{chroma4gray!80!black}{1} & \textcolor{chroma4gray!80!black}{2} & \textcolor{chroma4gray!80!black}{3} & \textcolor{chroma4gray!80!black}{4} & \textcolor{chroma4gray!80!black}{5} & \textcolor{chroma4gray!80!black}{6} \\
        \textcolor{chroma4gray!80!black}{1} & 0 & 4 & 4 & 4 & 3 & 3 \\
        \textcolor{chroma4gray!80!black}{2} & 4 & 0 & 4 & 4 & 3 & 3 \\
        \textcolor{chroma4gray!80!black}{3} & 4 & 4 & 0 & 4 & 3 & 3 \\
        \textcolor{chroma4gray!80!black}{4} & 4 & 4 & 4 & 0 & 3 & 3 \\
        \textcolor{chroma4gray!80!black}{5} & 3 & 3 & 3 & 3 & 0 & 3 \\
        \textcolor{chroma4gray!80!black}{6} & 3 & 3 & 3 & 3 & 3 & 0 
    \end{bmatrix}.
\end{equation*}
The partition $V=X_1\cup X_2$ with $X_1=\{1,2,3,4\}$ and $X_2=\{5,6\}$ yields the equitable quotient matrix 
\[
Q=\begin{bmatrix}
12 & 6 \\
12 & 3
\end{bmatrix}.
\]
From $X_1$ we obtain three eigenvalues $-4$ with linearly independent eigenvectors $e_1-e_2$, $e_2-e_3,$ $e_3-e_4$ and $X_2$ provides us with an eigenvalue $-3$ corresponding to the eigenvector $e_5-e_6$. A lower bound on the energy is therefore
\begin{equation*}
    2\sum_{\mathclap{i=1}}^k (|X_{i}|-1)@m(X_i) = 2@((4-1)@4 + (2-1)@3) = 30.
\end{equation*}
Theorem~\ref{thm:energyRefined} also tells us that this lower bound is not tight, because $Q$ has another negative eigenvalue, as its determinant is $-36$. However, the achieved value is quite close to the actual energy of $15+3\sqrt{41}\approx 34.21$.
\end{example*}

\section{Ultrametric distance matrices}\label{sec:ultrametric}

The structure of distance matrices of graphs is an active research topic. See Aouchiche and Hansen~\cite{aouchiche2014distance} for an overview or Stevanović and Indulal~\cite{stevanovic2009distance} for results on the distance energy. Most of the studies in that field involve shortest paths distances, whereas our results shed light on the spectral properties of ultrametric distance matrices. To see that, consider a graph $G$ and its edge-connectivity matrix $\Cm=(\cm_{vw})$ and define for two vertices $v,w\in V(G)$ the distance \mbox{$d(v,w)\coloneqq (\cm_{vw})^{-1}$}. Then the strong triangle inequality $d(x,z)\leq \max\{d(x,y), d(y,z)\}$ holds for all $x,y,z\in V(G)$. See also Gurvich~\cite{Gurvich10} for how this metric is some kind of resistance distance. Many of the results of the previous section have analogues in this setting. Indeed, if the entries of a matrix $D$ arise from an ultrametric, then $-D$ satisfies the \GH and all of our previous results that do not require a nonnegativity assumption hold for $-D$ as well. We state those results for completeness, occasionally with a slightly different wording.
\begin{theorem}
Let $d\colon V\times V\to \mathbb{R}$ be an ultrametric and denote the corresponding distance matrix by $D=(d(x,y))$. This in particular requires $d(x,x)=0$ for all $x\in V$. Furthermore, denote by $r(x)=\min\{d(x,y) : y\in V\setminus \{x\} \}$ the distance of $x$ to a nearest point. Then the following statements hold:
\begin{enumerate}
    \item The points $x,y \in V$ with $x\neq y$ are mutually nearest points (this means that $d(x,y)\le d(x,z)$ and $d(x,y)\le d(z,y)$ for each $z\in V\setminus\{x,y\}$), if and only if $e_x-e_y$ is an eigenvector of $D$ with corresponding eigenvalue $-d(x,y)$.
    \item The matrix $D$ induces an equivalence relation on $V$ by
    \[
    x \sim y :\Leftrightarrow x = y\text{ or }r(x) = r(y) = d(x,y).
    \]
    \item If $x_1\sim x_2$ and $y_1\sim y_2$, then $d(x_1,y_1)=d(x_2,y_2)$ or, equivalently, $d(x,y)$ depends only on the equivalence classes of $x$ and $y$, respectively.
    \item Let $X_1,\ldots, X_k$ be the equivalence classes with respect to the relation $\sim$ and denote by $r(X_i)$ the common value $r(x)$ of $x\in X_i$. Then the restrictions $D\rvert_{X_i\times X_j}$ of $D$ to $X_i\times X_j$ induce an equitable partition of $D$ with equitable quotient matrix $Q=(q_{ij})_{i,j=1,\ldots, k}$, where
    \[
    q_{ij} = \begin{cases}
        d(x,y)|X_j|     &\text{if }i\neq j,\text{ where }x\in X_i\text{ and }y\in X_j, \\ 
        r(X_j)(|X_j|-1) &\text{if }i=j.
    \end{cases}
    \]
\item The energies of $D$ and its equitable quotient matrix $Q$ are related by
\begin{equation*}
    \Ep(D) = \Ep(Q) + \textnormal{trace}(Q)  \ge 2\sum_{\mathclap{i=1}}^k (|X_{i}|-1)@r(X_i).
\end{equation*}
\end{enumerate}
\end{theorem}
We also have a lower bound on the smallest eigenvalue of an ultrametric distance matrix, which essentially relies on the following result of Zhan~\cite[Therorem 1]{Zhan2006} for general symmetric interval matrices.
\begin{theorem}\label{thm:zhan}
Let $X$ be a real symmetric $V\times V$ matrix with entries in an interval~$[m,M]$. Denote by $\lambda_n(X)$ the smallest eigenvalue of $X$ and consider the problem to
\begin{equation}
    \text{minimize }\lambda_n(X)\text{ such that }m@J \leq X \leq M@J,\tag{P1}
\end{equation}
where the inequalities are meant componentwise. Then the only optimal solution of problem~(P1) up to simultaneous permutations of rows and columns is
\begin{equation*}
    X=X^*\coloneqq\begin{bmatrix}
    m@J_{V_1\times V_1} & M@J_{V_1\times V_2}\\
    M@J_{V_2\times V_1} & m@J_{V_2\times V_2}
\end{bmatrix},
\end{equation*}
where $V_1 \cup V_2=V$ is a bipartition of $V$ with $\left \lvert |V_1|-|V_2|  \right\rvert \le 1$. The optimal value is
\begin{flalign*}
    \hspace{12ex}\lambda_n(X^*) &= \begin{cases}
    n@(m-M)/2                                      &\text{if }n \text{ is even,}\\
    \big(n@m- \sqrt{m^2+(n^2-1)@M^2}@@\big)/2    &\text{if }n \text{ is odd.}
\end{cases}&&
\end{flalign*}
\end{theorem}
\begin{theorem}
Let $d\colon V\times V\to \mathbb{R}$ be an ultrametric and let $D=(d(x,y))$ be the corresponding distance matrix. Denote \mbox{$m\coloneqq\min\{d(x,y) : x,y\in V\text{ with }x\neq y\}$}, $M\coloneqq\max\{d(x,y) : x,y\in V\text{ with }x\neq y\}$, and $n\coloneqq|V|$. Then the smallest eigenvalue $\lambda_n(D)$ of $D$ satisfies
\begin{flalign*}
\hspace{12ex}\lambda_n(X^*) \geq
\begin{cases}
    n@(m-M)/2-m                                      &\text{if }n \text{ is even,}\\
    \big(n@m- \sqrt{m^2+(n^2-1)@M^2}@@\big)/2 - m    &\text{if }n \text{ is odd.}
\end{cases}&&
\end{flalign*}
This bound is attained if and only if there is a bipartition $V=V_1 \cup V_2$ of $V$ with $\left \lvert |V_1|-|V_2|  \right\rvert \le 1$ such that the distance between each pair $(x,y)$ with $x\neq y$ is
\begin{flalign*}
\hspace{12ex}\hphantom{\lambda_n(X^*) \geq x}\negphantom{d(x,y)=x}d(x,y)=\begin{cases}
    m   &\text{ if }(x,y)\in V_1\times V_1 \text{ or }(x,y)\in V_2 \times V_2, \\
    M   &\text{ if }(x,y)\in V_1\times V_2 \text{ or }(x,y)\in V_2 \times V_1.
\end{cases}&&
\end{flalign*}
Equivalently, the stated bound is attained if and only if $D=X^*-mI$, where $X^*$ is defined as in Theorem~\ref{thm:zhan}.
\end{theorem}
\begin{proof} We recall first that if $d$ is an ultrametric, then $d(x,x) = 0$ for all $x\in V$. Consequently, $D$ has only zeros on its diagonal. Thus the optimal value of the problem to
\begin{equation}
    \text{minimize }\lambda_n(X)\text{ such that }m@(J-I) \leq X \leq M@(J-I)\tag{P2}
\end{equation}
is certainly a lower bound for the smallest eigenvalue of an ultrametric matrix.
Now, if $X$ is an optimal solution of problem~(P2), then $X+mI$ is a feasible solution of problem~(P1) and thus $\text{opt(P2)}+m\geq\text{opt(P1)}$. Conversely, by Theorem~\ref{thm:zhan}, $X^* + mI$ is an optimal solution of problem~(P1) and $X^*$ is feasible for problem~(P2). Consequently, $\text{opt(P1)}-m\geq\text{opt(P2)}$, as desired. The stated bound on the smallest eigenvalue follows directly from Theorem~\ref{thm:zhan} and the fact that the unique optimal solution to (P2) is an ultrametric matrix.
\end{proof}

\section{Conclusions and related problems}
Motivated by the recent interest in the spectra of vertex-connectivity matrices, we found even stronger structural properties for the edge-connectivity matrix. Note that both problems, the issue of whether the energy of a vertex-connectivity matrix of a graph on $n$ vertices is bounded from above by $2(n-1)^2$ as well as the question for a tight lower bound on its smallest eigenvalue, are still open.

Furthermore, we obtained that for edge-connectivity matrices the upper bound $(n-1)^2$ on the largest eigenvalue as well as the lower bound $-(n-1)$ on the smallest eigenvalue are attained simultaneously by the matrix that arises from the complete graph. This provides us with a tight upper bound $(n-1)(n-2)$ on the \emph{spread} of such matrices, which is defined as the largest distance between any two eigenvalues of a matrix. This resolves a special case of an intriguing open problem stated by Zhan~\cite[Problem 2]{Zhan2006} for which Fallat and Xing~\cite{fallat2012spread} formulated a detailed conjecture. There, the question for the spread is stated for general symmetric interval matrices. Note, however, that our result on the spread is bound to the very rich structure of edge-connectivity matrices. In different settings, one might not expect to find a matrix that attains an upper bound on the largest eigenvalue and a lower bound on the smallest eigenvalue simultaneously.

\section*{Acknowledgments}
We greatly thank Dragan Stevanovi{\'c} for drawing our attention to the questions about vertex-connectivity matrices. This inspired our work on edge-connectivity matrices. Our research was partially funded by the Deutsche Forschungsgemeinschaft (DFG, German Research Foundation) -- Project-ID 416228727 -- SFB 1410.

\bibliographystyle{plain}
\bibliography{bibliography}

\begin{thebibliography}{10}

\bibitem{akbari2020edge}
Saieed Akbari, Seyran Azizi, Modjtaba Ghorbani, and Xueliang Li.
\newblock On edge-path eigenvalues of graphs.
\newblock {\em Linear and Multilinear Algebra}, pages 1--11, 2020.

\bibitem{aouchiche2014distance}
Mustapha Aouchiche and Pierre Hansen.
\newblock Distance spectra of graphs: A survey.
\newblock {\em Linear Algebra and its Applications}, 458:301--386, 2014.

\bibitem{diestel2017graph}
Reinhard Diestel.
\newblock {\em Graph Theory}.
\newblock Springer, 2017.

\bibitem{fallat2012spread}
Shaun~M. Fallat and YongJun Xing.
\newblock On the spread of certain normal matrices.
\newblock {\em Linear and Multilinear Algebra}, 60(11-12):1391--1407, 2012.

\bibitem{gomoryhu1961multi}
Ralph~E. Gomory and Tien~Chung Hu.
\newblock Multi-terminal network flows.
\newblock {\em SIAM Journal}, 9(4):551--570, 1961.

\bibitem{Gurvich10}
Vladimir Gurvich.
\newblock Metric and ultrametric spaces of resistances.
\newblock {\em Discrete Applied Mathematics}, 158(14):1496--1505, 2010.

\bibitem{uniformly2021}
Frank Göring, Tobias Hofmann, and Manuel Streicher.
\newblock Uniformly connected graphs.
\newblock {\em Preprint}, 2021.

\bibitem{horn1990matrix}
Roger~A. Horn and Charles~R. Johnson.
\newblock {\em Matrix Analysis}.
\newblock Cambridge University Press, 1990.

\bibitem{ilic2019path}
Aleksandar Ili{\'c} and Milan Ba{\v{s}}i{\'c}.
\newblock Path matrix and path energy of graphs.
\newblock {\em Applied Mathematics and Computation}, 355:537--541, 2019.

\bibitem{koolen2001maximal}
Jack~H. Koolen and Vincent Moulton.
\newblock Maximal energy graphs.
\newblock {\em Advances in Applied Mathematics}, 26(1):47--52, 2001.

\bibitem{gutman2012energy}
Xueliang Li, Yongtang Shi, and Ivan Gutman.
\newblock {\em Graph Energy}.
\newblock Springer, 2012.

\bibitem{menger1927allgemeinen}
Karl Menger.
\newblock {Zur allgemeinen Kurventheorie}.
\newblock {\em Fundamenta Mathematicae}, 10(1):96--115, 1927.

\bibitem{patekar2020signless}
Prashant~P. Patekar and Maruti~M. Shikare.
\newblock On the path cospectral graphs and path signless {L}aplacian matrix of
  graphs.
\newblock {\em Journal of Mathematical and Computational Science},
  10(4):922--935, 2020.

\bibitem{shikare2018path}
Maruti~M. Shikare, Prashant~P. Malavadkar, Shridhar~C. Patekar, and Ivan
  Gutman.
\newblock On path eigenvalues and path energy of graphs.
\newblock {\em MATCH Communications in Mathematical and in Computer Chemistry},
  79:387--398, 2018.

\bibitem{stevanovic2009distance}
Dragan Stevanovi{\'c} and Gopalapillai Indulal.
\newblock The distance spectrum and energy of the compositions of regular
  graphs.
\newblock {\em Applied Mathematics Letters}, 22(7):1136--1140, 2009.

\bibitem{you2019spectrum}
Lihua You, Man Yang, Wasin So, and Weige Xi.
\newblock On the spectrum of an equitable quotient matrix and its application.
\newblock {\em Linear Algebra and its Applications}, 577:21--40, 2019.

\bibitem{Zhan2006}
Xingzhi Zhan.
\newblock Extremal eigenvalues of real symmetric matrices with entries in an
  interval.
\newblock {\em SIAM Journal on Matrix Analysis and Applications},
  27(3):851--860, 2006.

\end{thebibliography}

\end{document}